\documentclass[11pt]{amsart}
\usepackage{amsmath,amsfonts,latexsym,graphicx,amssymb,url}
\usepackage{hyperref,pdfsync}
\usepackage{amsmath,txfonts,pifont,bbding,pxfonts,manfnt}
\usepackage[active]{srcltx}
\usepackage{wasysym,pstricks}
\usepackage{tikz}
\usepackage{xcolor}
\usepackage{soul}

\newcommand{\diam}{\text{{\rm diam}}}

\newcommand{\R}{{\mathbb R}} 

\newcommand{\Z}{{\mathbb Z}}
\newcommand{\N}{{\mathbb N}}

\renewcommand{\(}{\left(}
\renewcommand{\)}{\right)}
\def \e {\varepsilon}

\newtheorem{theorem}{Theorem}
\newtheorem{corollary}[theorem]{Corollary}
\newtheorem{lemma}[theorem]{Lemma}

\newtheorem{remark}[theorem]{Remark}

\begin{document}











\title[$h$-Monotone maps ~~~\today]{Fine properties of monotone maps arising in optimal transport for non quadratic costs}
\author[C. E. Guti\'errez and A. Montanari, \today]{Cristian E. Guti\'errez and Annamaria Montanari
\\
\today}
\thanks{
The first named author wishes to warmly thank Annamaria Montanari for the invitation to visit the University of Bologna during June/July 2022 where this research was partially carried out. He would also like to thank this institution and the Indam from the hospitality and support.
C.E.G. was partially supported by NSF grant DMS--1600578, and A. M. was partially supported by a grant from GNAMPA}
\address{Department of Mathematics\\Temple University\\Philadelphia, PA 19122}
\email{cristian.gutierrez@temple.edu}
\address{Dipartimento di Matematica\\Piazza di Porta San Donato 5\\Universit\`a di Bologna\\40126 Bologna, Italy}
\email{annamaria.montanari@unibo.it}
\dedicatory{To our dear friend and collaborator Ermanno Lanconelli on his 80th birthday\\
celebrating your remarkable journey and contributions to mathematics.\\
With warmest regards and deepest appreciation.\\
}

\maketitle

\begin{abstract}

The cost functions considered are $c(x,y)=h(x-y)$, where $h\in C^2(\mathbb{R}^n)$, homogeneous of degree $p\geq 2$, with a positive definite Hessian in the unit sphere. We study multivalued monotone maps with respect to that cost and establish that they are single-valued almost everywhere.
Further consequences are then deduced.

\end{abstract}
\tableofcontents
\setstcolor{blue}

\setcounter{equation}{0}
\section{Introduction}

In this paper, we examine monotone maps relative to cost functions that aren't necessarily quadratic, but arise in optimal transport. Our aim is to delve into this concept to establish fine properties of these maps, particularly their almost everywhere single-valued nature.

This investigation is crucial for proving $L^\infty$-estimates for these maps, as they can be derived through the integration of the inequality \eqref{eq:map T is c-monotone h}, a method utilized in \cite{Gutierrez-Montanari:Linfty-estimates}. Furthermore, it leads to differentiability properties of the maps, which we will address elsewhere.

This work complements our paper \cite{Gutierrez-Montanari:Linfty-estimates}, which stems from the foundational work by Goldman and Otto \cite{2020-goldman-otto-variational}, who developed a variational approach to establish regularity of optimal maps for quadratic costs.

The cost functions considered have the form $c(x,y)=h(x-y)$ where $h\in C^2(\R^n)$, is nonnegative, positively homogeneous of degree $p$ for some $p\geq 2$, and satisfies \eqref{eq:strict ellipticity of D2h}.

We say that a multivalued map $T:\R^n\to \mathcal P(\R^n)$ is $c$-monotone (or $h$-monotone) if
\begin{equation}\label{eq:map T is c-monotone h}
c\(x,\xi\)+ c\(y,\zeta\)\leq c\(x,\zeta\)+ c\(y,\xi\),
\end{equation}
for all $x,y\in \text{dom}(T)$ and for all $\xi\in T(x)$ and $\zeta\in T(y)$; where 
$\text{dom}(T)=\{x\in \R^n:T(x)\neq \emptyset\}$.
For the quadratic cost $h(x)=|x|^2$, this corresponds to the well-known notion of monotone map, extensively studied in various contexts as described in the classic book by H. Br\'ezis \cite{brezis-book-monotone-maps}.

In the framework of optimal transport theory, if $c(x,y):D\times D^*\to [0,+\infty)$ is a general cost function, then 
the optimal map for the Monge problem is given by $T=\mathcal N_{c,\phi}$ where $\phi$ is $c$-concave and 
\[
\mathcal N_{c,\phi}(x)
=
\left\{m\in D^*:\phi(x)+\phi^c(m)=c(x,m)\right\}
\]
with $\phi^c(m)=\inf_{x\in D}\(c(x,m)-\phi(x)\)$, see for example \cite[Chapter 6]{Gutierrez:23}.
It is then clear that the optimal map $\mathcal N_{c,\phi}$ is $c$-monotone, and further it is $c$-cyclically monotone.  That is, a map $T:D\to \mathcal P(D^*)$ is $c$-cyclically monotone if for all $N\in \N$ and $\{(x_i,\xi_i)\}_{i=1}^N$ with $\xi_i\in T(x_i)$, and for any permutation $\sigma$ of the indices $1,\cdots ,N$, we have
\begin{equation}
 \sum_{i=1}^Nc\(x_i, \xi_i\)\leq \sum_{i=1}^Nc\(x_i, \xi_{\sigma(i)}\).
 \end{equation}
These three notions—monotonicity, cyclic monotonicity, and optimality—are interrelated, and the broadest class of maps consists of those that are just monotone. An examination of these notions with respect to a cost is presented in the recent paper \cite{2024-de-pascale-survey-monotonicity}, which also includes several examples and extensions.

It's important to note that we will only consider maps $T$ satisfying \eqref{eq:map T is c-monotone h}, which are not necessarily optimal or cyclically monotone.


The organization of the paper is outlined as follows:

Section \ref{sec:preliminaries} presents equivalent formulations of $h$-monotonicity, akin to the standard monotone map concept, which are required for later discussions.

Our main result, Theorem \ref{thm:monotoneHntoR2n+1}, asserts that every $h$-monotone map $T$ is single-valued almost everywhere, and hence ensuring the validity of the inequality \eqref{eq:map T is c-monotone h} for almost all points $(x,Tx)$, $x\in \text{dom}(T)$. This is critical for establishing $L^\infty$-estimates for $h$-monotone maps, as elucidated in \cite{Gutierrez-Montanari:Linfty-estimates}. Furthermore, Theorem \ref{thm:monotoneHntoR2n+1} also implies that $T$ induces a push-forward measure, as illustrated in Corollary \ref{cor:aleksandrovlemma}. Moreover, if $T$ is maximal, it induces an Aleksandrov-type measure, as indicated in Theorem \ref{thm:monge ampere measure for h monotone maximal}.

Section \ref{sec:rectifiability} demonstrates that if $T$ satisfies \eqref{eq:map T is c-monotone h}, then the set $S=\{(x,\xi):\xi\in Tx, x\in \text{dom}(T)\}$ is rectifiable.

Finally, the Appendix (Section \ref{sec:appendix}) contains a result utilized in the proof of Theorem \ref{thm:monotoneHntoR2n+1}.

\section{Preliminaries on $h$-Monotone maps}\label{sec:preliminaries}
\setcounter{equation}{0}

In this section we present equivalent formulations of the Definition \eqref{eq:map T is c-monotone h} of $h$-monotonicity and integral representation formulas that will be needed to prove our results.
From \eqref{eq:map T is c-monotone h} and since $h\in C^2$ we have
\begin{align*}
0&\leq  h(y-\xi)-h(y-\zeta)-\(h(x-\xi)-h(x-\zeta)\)\notag\\
&=\int_0^1 \langle Dh(y-\zeta+s(\zeta-\xi)), \zeta-\xi\rangle ds -\int_0^1 \langle Dh(x-\zeta+s(\zeta-\xi)), \zeta-\xi\rangle ds\notag\\
&=\int_0^1 \langle Dh(y-\zeta+s(\zeta-\xi))-Dh((x-\zeta+s(\zeta-\xi)), \zeta-\xi\rangle ds\notag\\
&= \int_0^1\int_0^1\langle D^2h(y-\zeta+s(\zeta-\xi)+t (x-y)) (x-y),  (\xi-\zeta)\rangle dt\, ds\notag\\
&=  \langle A(x,y;\xi,\zeta) (x-y),  \xi-\zeta\rangle,\qquad \forall x,y\in \text{dom}(T), \xi\in T(x),\zeta\in T(y).
\end{align*}

Therefore \eqref{eq:map T is c-monotone h} is equivalent to 
\begin{equation}\label{eq:A}
 \langle A(x,y;\xi,\zeta) (x-y),  \xi-\zeta\rangle\geq 0,\qquad \forall x,y\in \text{dom}(T), \xi\in T(x),\zeta\in T(y)
\end{equation}
with
\begin{equation}\label{eq:definition of A(x,y)}
A(x,y;\xi,\zeta)=\int_0^1\int_0^1 D^2h(y-\zeta+s(\zeta-\xi)+t (x-y))dt\, ds.
\end{equation}
The matrix $A(x,y;\xi,\zeta)$ is clearly symmetric, and satisfies $A(x,y;\xi,\zeta)=A(y,x;\zeta,\xi)$ by making the change of variables $t=1-t',s=1-s'$ in the integral.
If $h$ is homogenous of degree $p$ with $p\geq 2$, then $D^2h(z)$ is positively homogeneous of degree $p-2$.
We assume that $D^2h(x)$ is positive definite for each $x\in S^{n-1}$ and since $h\in C^2$, then there are positive constants $\lambda,\Lambda$ such that 
\begin{equation}\label{eq:strict ellipticity of D2h}
\lambda \,|v|^2
\leq 
\left\langle D^2h(x)v,v\right\rangle
\leq
\Lambda \,|v|^2,\qquad \forall x\in S^{n-1},v\in \R^n.
\end{equation}
We then have
\[
A(x,y;\xi,\zeta)=
\int_0^1 \int_0^1|y-\zeta+s(\zeta-\xi)+t (x-y)|^{p-2}D^2h\left(\frac{y-\zeta+s(\zeta-\xi)+t (x-y)}{ |y-\zeta+s(\zeta-\xi)+t (x-y)|}\right)dt \, ds
\]
and
\begin{equation}\label{eq:ellipticity of A}
\lambda \,\Phi(x,y;\xi,\zeta)\,|v|^2
\leq
\left\langle  A(x,y;\xi,\zeta)\,v,v\right\rangle\leq \Lambda \, \Phi(x,y;\xi,\zeta)\,|v|^2\quad \forall v \in \R^n,
\end{equation}
with 
\begin{equation}\label{eq:Phi}
 \Phi(x,y;\xi,\zeta)=\int_0^1\int_0^1|y-\zeta+s(\zeta-\xi)+t (x-y)|^{p-2}dt \, ds.
\end{equation}
We also have that $\Phi(x,y;\xi,\zeta)=0$ if and only if $y-\zeta+s(\zeta-\xi)+t (x-y)=0$ for all $s,t\in [0,1]$. That is, $\Phi(x,y;\xi,\zeta)=0$ if and only if $y-\zeta=0$, $\zeta-\xi=0$ and $x-y=0$.
Therefore $\Phi(x,y;\xi,\zeta)>0$ if and only if $\zeta\neq y$ or $\zeta\neq \xi$ or $x\neq y$.

\section{$h$-monotone maps are single valued a.e.}
\setcounter{equation}{0}

The main result of this section is the following.

\begin{theorem}\label{thm:monotoneHntoR2n+1}
If $T:\R^n\to \mathcal P(\R^n)$ is a multivalued map that is $h$-monotone, with $h$ satisfying the assumptions in the previous section, then $T(x)$ is a singleton for all points $x\in \text{dom}(T)$ except on a set of measure zero.
\end{theorem}
\begin{proof}
Let 
\[
S=\{x\in \text{dom}(T): T(x) \text{ is not a singleton}\},
\]
we will prove that $S$ has Lebesgue measure zero.
For each $k\geq 1$ integer, let
\[
S_k=\{x\in \text{dom}(T): \diam(T(x)) > 1/k\}.
\]
We have $S=\cup_{k=1}^\infty S_k$.
Also for each $k\geq 1$ integer, let $\{B_j\}_{j=1}^\infty$ be a family of Euclidean balls in $\R^n$ each one with radius $\epsilon/k$ with $\epsilon>0$ small to be chosen later after inequality \eqref{eq:lower bound for Ftheta}, depending only $\Lambda/\lambda$, the constants in \eqref{eq:strict ellipticity of D2h}, and such that $\R^n=\cup_{j=1}^\infty B_j$.
Let
\[
B_j^*=\{x\in \text{dom}(T):T(x)\cap B_j\neq \emptyset\},
\text{ and }
S_{kj}=S_k\cap B_j^*.
\]
We have $\cup_{j=1}^\infty B_j^*= \text{dom}(T)$ and so $S=\cup_{k,j=1}^\infty S_{kj}$.
We shall prove that $|S_{kj}|=0$ for all $k$ and $j$.
In order to do this we recall that a point $x\in \R^n$ is a density point for the set $E\subset \R^n$, with $E$ not necessarily Lebesgue measurable, if 
\[
\limsup_{r\to 0}\dfrac{|E\cap B_r(x)|_*}{|B_r(x)|}=1
\]
where $|\cdot |_*$ denotes the Lebesgue outer measure and $B_r(x)$ is the Euclidean ball with center $x$ and radius $r$.
We shall prove in Lemma \ref{lm:densitypointslebesgue}, that if $E\subset \R^n$ is any set, then almost all points in $E$ are density points for $E$.
In view of this, if we prove that each $x_0\in S_{kj}$ is not a density point for $S_{kj}$, then it follows that $|S_{kj}|=0$\footnote{A priori we do not know if the set $S_{kj}$ is Lebesgue measurable and therefore we cannot apply the Lebesgue differentiation theorem directly.}.

Let us fix $x_0\in S_{kj}$. Then $\diam T(x_0)> 1/k$ and $T(x_0)\cap B_j\neq \emptyset$, so we can pick $y_1\in T(x_0)\cap B_j$. There exists $y_2\in T(x_0)$ such that $|y_1-y_2|\geq 1/2k$,
because otherwise $T(x_0)\subset B_{1/2k}(y_1)$ which would imply that $\diam T(x_0)\leq 1/k$. 
Also notice that $y_2\notin B_j$ if $\epsilon<1/4$, since $B_j$ has radius $\epsilon/k$.
In addition, if $x_j$ is the center of $B_j$, it follows that $|x_j-y_2|\geq |y_2-y_1|-|y_1-x_j|\geq \dfrac{1}{2k}-\dfrac{\epsilon}{k}(>0)$ if $\epsilon<1/2$. 
And so for each $\xi\in B_j$, $|\xi-y_2|\geq |x_j-y_2|-|\xi-x_j|\geq \dfrac{1}{2k}-\dfrac{2\epsilon}{k}\geq \dfrac{2\epsilon}{k}$ if $\epsilon<1/8$.

Set
\[
y_2-y_1=e.
\]
Given $x\neq x_0$, $x\in \text{dom}(T)$, let $z=x-x_0-\(\(x-x_0\)\cdot \dfrac{e}{|e|}\)\,\dfrac{e}{|e|}$.
Then $z\cdot \dfrac{e}{|e|}=0$, $\(x-x_0\)\cdot \dfrac{z}{|z|}=|z|$, and
\[
x-x_0=\(\(x-x_0\)\cdot \dfrac{e}{|e|}\)\,\dfrac{e}{|e|}+\(\(x-x_0\)\cdot \dfrac{z}{|z|}\)\,\dfrac{z}{|z|}.
\]
If $\delta$ is the angle between the unit vectors $\dfrac{x-x_0}{|x-x_0|}$ and $\dfrac{e}{|e|}$, then we have $0<\delta<\pi$ and 
\begin{equation}\label{eq:writing of vector x-x_0}
\dfrac{x-x_0}{|x-x_0|}
=\cos \delta\,\dfrac{e}{|e|}+\sin \delta\,\dfrac{z}{|z|}.
\end{equation}
Let $\xi \in T(x)$ and consider the matrix $A(x,x_0;\xi,y_2)$ defined by \eqref{eq:definition of A(x,y)}. 
From \eqref{eq:ellipticity of A},
\begin{equation}\label{eq:ellipticity bis}
\lambda\,\Phi(x,x_0;\xi,y_2) \,Id\leq A(x,x_0;\xi,y_2)\leq \Lambda\,\Phi(x,x_0;\xi,y_2) \,Id
\end{equation}
and since $x\neq x_0$, $\Phi(x,x_0;\xi,y_2)>0$. To simplify the notation we write $A(x,x_0;\xi,y_2)=A$ and $\Phi(x,x_0;\xi,y_2)=\Phi$.

To show that $x_0$ is not a density point for $S_{kj}$ we analyze the sizes of the angles between various vectors using the $h$-monotonicity.  For the sake of clarity we divide the proof into three steps.

{\bf Step 1.}  We shall estimate
$
F(\delta):=\text{angle}\(A^{1/2}\dfrac{x-x_0}{|x-x_0|}, A^{1/2}\dfrac{e}{|e|}\),
$
proving \eqref{eq:estimate for F(delta)}.

We have 
$
\cos F(\delta)=\dfrac{\left\langle 
A^{1/2}\dfrac{x-x_0}{|x-x_0|}, A^{1/2}\dfrac{e}{|e|}\right\rangle}{\left|  A^{1/2}\dfrac{x-x_0}{|x-x_0|} \right|\,\left|  A^{1/2}\dfrac{e}{|e|} \right|},
$
and from \eqref{eq:writing of vector x-x_0}
\begin{align*}
\left\langle 
A^{1/2}\dfrac{x-x_0}{|x-x_0|}, A^{1/2}\dfrac{e}{|e|}\right\rangle
&=
\cos \delta \,\left| A^{1/2}\dfrac{e}{|e|}\right|^2
+
\sin \delta \,
\left\langle 
A^{1/2}
\dfrac{z}{|z|}, A^{1/2}\dfrac{e}{|e|}\right\rangle,\\
\left\langle 
A^{1/2}\dfrac{x-x_0}{|x-x_0|}, A^{1/2} \dfrac{x-x_0}{|x-x_0|}\right\rangle
&=
\cos^2 \delta \,\left| A^{1/2}\dfrac{e}{|e|}\right|^2
+
2\,\sin \delta \,\cos \delta\,
\left\langle 
A^{1/2}
\dfrac{z}{|z|}, A^{1/2}\dfrac{e}{|e|}\right\rangle
+
\sin^2 \delta \,\left| A^{1/2}\dfrac{z}{|z|}\right|^2.
\end{align*}
Hence
\begin{align}\label{eq:split formula for cos F}
\cos F(\delta)
&=
\begin{cases}
\dfrac{1+\tan \delta \,
\dfrac{\left\langle 
A^{1/2}
\dfrac{z}{|z|}, A^{1/2}\dfrac{e}{|e|}\right\rangle}{\left| A^{1/2}\dfrac{e}{|e|}\right|^2}}{\sqrt{1+\tan^2 \delta\,
\dfrac{\left| A^{1/2}\dfrac{z}{|z|}\right|^2}{\left| A^{1/2}\dfrac{e}{|e|}\right|^2}
+
2\,\tan \delta\,\dfrac{\left\langle 
A^{1/2}
\dfrac{z}{|z|}, A^{1/2}\dfrac{e}{|e|}\right\rangle}{\left| A^{1/2}\dfrac{e}{|e|}\right|^2}}} & \text{for $0<\delta< \pi/2$}\\
-\dfrac{1+\tan \delta \,
\dfrac{\left\langle 
A^{1/2}
\dfrac{z}{|z|}, A^{1/2}\dfrac{e}{|e|}\right\rangle}{\left| A^{1/2}\dfrac{e}{|e|}\right|^2}}{\sqrt{1+\tan^2 \delta\,
\dfrac{\left| A^{1/2}\dfrac{z}{|z|}\right|^2}{\left| A^{1/2}\dfrac{e}{|e|}\right|^2}
+
2\,\tan \delta\,\dfrac{\left\langle 
A^{1/2}
\dfrac{z}{|z|}, A^{1/2}\dfrac{e}{|e|}\right\rangle}{\left| A^{1/2}\dfrac{e}{|e|}\right|^2}}} & \text{for $\pi/2<\delta< \pi$.}\end{cases}
\end{align}
From \eqref{eq:ellipticity bis}
\[
\dfrac{\lambda}{\Lambda}=\dfrac{\lambda\,\Phi}{\Lambda\,\Phi}\leq C:=\dfrac{\left| A^{1/2}\dfrac{z}{|z|}\right|^2}{\left| A^{1/2}\dfrac{e}{|e|}\right|^2}
\leq
\dfrac{\Lambda\,\Phi}{\lambda\,\Phi}=\dfrac{\Lambda}{\lambda}.
\]
If 
\[
B:=\dfrac{\left\langle 
A^{1/2}
\dfrac{z}{|z|}, A^{1/2}\dfrac{e}{|e|}\right\rangle}{\left| A^{1/2}\dfrac{e}{|e|}\right|^2},
\]
then by Cauchy-Schwarz
\[
|B|\leq 
\dfrac{\left| A^{1/2}\dfrac{z}{|z|}\right|}{\left| A^{1/2}\dfrac{e}{|e|}\right|}
=
\sqrt{C}
\leq
\sqrt{\dfrac{\Lambda}{\lambda}}.
\]
Setting 
\begin{equation}\label{eq:definition of g}
g(s)=
\begin{cases}
\dfrac{1+B\,s }{\sqrt{1+C\,s^2 +2\,B\,s}} & \text{if $s\in(0,\infty)$}\\
\\
-\dfrac{1+B\,s }{\sqrt{1+C\,s^2 +2\,B\,s}} & \text{if $s\in (-\infty,0)$}
\end{cases}
\end{equation}
we get from \eqref{eq:split formula for cos F} that 
\[
\cos F(\delta)=g(\tan \delta).
\]
Notice that since $|B|\leq \sqrt{C}$ we have $1+C\,s^2 +2\,B\,s\geq 0$ for all $s\in \R$.
If $|B|< \sqrt{C}$, then $1+C\,s^2 +2\,B\,s> 0$ for all $s\in \R$; and if $|B|= \sqrt{C}$, then 
$1+C\,s^2 +2\,B\,s=1+B^2\,s^2 +2\,B\,s=(1+B\,s)^2$ which is strictly positive except when $s=-1/B$.
For $s>0$, let us now estimate 
\begin{equation*}
\Delta(s)=1-g(s)
=
\dfrac{\(C-B^2\)\,s^2}{\(1+B\,s+\sqrt{1+C\,s^2 +2\,B\,s}\)\sqrt{1+C\,s^2 +2\,B\,s}}.
\end{equation*}
Since $-\sqrt{C}\leq B\leq \sqrt{C}$, it follows that 
\[
1+C\,s^2 +2\,B\,s\geq \(1-\sqrt{C}\,|s|\)^2
\]
for $-\infty<s<\infty$. Therefore, if $|s|\leq 1/\(2\sqrt{C}\)$, then we get $1-\sqrt{C}\,|s|\geq 1/2$.
Also $1+B\,s\geq 1-\sqrt{C}\,s$ if $s>0$, and $1+B\,s\geq 1+\sqrt{C}\,s$ for $s<0$. So
$1+B\,s\geq 1-\sqrt{C}\,|s|$. 
Hence 
\[
0\leq \Delta(s)\leq 2\,\(C-B^2\)\,s^2\qquad \text{for $0<s\leq 1/\(2\sqrt{C}\)$}.
\]
In particular, 
since $\lambda/\Lambda\leq C\leq \Lambda/\lambda$ and $|B|\leq \sqrt{C}$, 
we obtain the bound
\[
0\leq \Delta(s)\leq 2\,\dfrac{\Lambda}{\lambda}\,s^2\qquad \text{for $0<s\leq  (1/2)\sqrt{\lambda/\Lambda}$}.
\]
This implies that 
\[
1-\cos F(\delta)\leq 4\,\dfrac{\Lambda}{\lambda}\,\tan^2 \delta\]
for $\delta$ such that $0<\tan \delta \leq (1/2)\sqrt{\lambda/\Lambda}$. Since $\tan \delta\sim \delta$ for $\delta$ sufficiently small, we then obtain the estimate
\[
1-\cos F(\delta)\leq 8\,\dfrac{\Lambda}{\lambda}\,\delta^2\]
for all $0<\delta\leq \delta_0$, with $\delta_0$ a positive number depending only on $\lambda/\Lambda$. Hence
\[
F(\delta)\leq \arccos \(1-8\,\dfrac{\Lambda}{\lambda}\,\delta^2\)\qquad \text{for $0<\delta\leq \delta_0$.}\]
On the other hand, $\dfrac{\arccos (1-x^2)}{x}\to \sqrt2$ as $x\to 0^+$, it follows that
\begin{equation}\label{eq:estimate for F(delta)}
F(\delta)\leq C(\Lambda/\lambda)\,\delta \qquad \text{for $0<\delta\leq \delta_0$.}
\end{equation}

Therefore, from the definitions of $\delta$ and $F(\delta)$ we obtain
\begin{equation}\label{eq:estimate for F(delta) rewritten}
\text{angle}\(A(x,x_0;\xi,y_2)^{1/2}(x-x_0),A(x,x_0;\xi,y_2)^{1/2}e\)
\leq
C(\Lambda/\lambda)\,\text{angle} (x-x_0,e),
\end{equation}
for all $\xi\in T(x)$ if $\text{angle} (x-x_0,e)\leq \delta_0$, $x\neq x_0$, $x\in \text{dom}(T)$.
On the other hand, from the monotonicity \eqref{eq:A}
\[
\langle A(x,x_0;\xi,y_2) (x-x_0),  \xi-y_2\rangle\geq 0,\qquad \forall \xi\in T(x)
\]
that is, $\langle A(x,x_0;\xi,y_2)^{1/2} (x-x_0),  A(x,x_0;\xi,y_2)^{1/2}(\xi-y_2)\rangle\geq 0$ for all $\xi\in T(x)$ and therefore
\[
\text{angle}\(A(x,x_0;\xi,y_2)^{1/2} (x-x_0),  A(x,x_0;\xi,y_2)^{1/2}(\xi-y_2)\)\leq \pi/2,\,\forall \xi\in T(x).
\]
Hence from \eqref{eq:estimate for F(delta) rewritten}
\begin{align}\label{eq:upper estimate of angle}
&\text{angle}\(A(x,x_0;\xi,y_2)^{1/2} (\xi-y_2),  A(x,x_0;\xi,y_2)^{1/2}e\)\notag\\
&\qquad \leq \pi/2+C(\Lambda/\lambda)\,\text{angle} (x-x_0,e),\,\forall \xi\in T(x),
\end{align} 
when $\text{angle} (x-x_0,e)\leq \delta_0$, $x\neq x_0$, $x\in \text{dom}(T)$.

{\bf Step 2.}
Let $\Gamma=\{x:\text{angle} (x-x_0,e)\leq \delta_0\}$.
We shall prove that 
\begin{equation}\label{eq:intersection empty}
T(\Gamma)\cap B_j=\emptyset,
\end{equation}
for $\epsilon$ sufficiently small depending only on $\Lambda/\lambda$; $B_j=B_{\epsilon/k}(x_j)$.

Recall that $x_0\in S_{kj}$, $y_1\in Tx_0\cap B_j$, and $y_2\in Tx_0$ with $|y_2-y_1|\geq 1/2k$.
Let $\mathcal C$ be the cone with vertex $y_2$, axis $e=y_2-y_1$, and opening $\pi+\theta_0$.
We have $\mathcal C\cap B_j=\emptyset$ when $\theta_0>0$ is small, for all $\epsilon<1$.
Suppose by contradiction that \eqref{eq:intersection empty} does not hold, that is, there is $\xi \in T(x)\cap B_j$ for some $x\in \Gamma$, and  
let $\theta$ be the angle between $\xi-y_2$ and $e$.
Since $\xi\in B_j$, we then have $\theta\geq \(\pi+\theta_0\)/2$ for all $\epsilon<1$.
To obtain a contradiction, we first right down the left hand side of \eqref{eq:upper estimate of angle} in terms of the angle $\theta$ and will show that it is close to $\pi$ when $\theta\sim \pi$, for $\epsilon$ sufficiently small, contradicting \eqref{eq:upper estimate of angle}.
\vskip 0.3in
\begin{figure}[h]
\begin{tikzpicture}
\draw (0,0) circle (3);
\draw (0,0) circle (1.5);
\draw[<-] (-.5,.6) -- (10,0);
\draw[<-] (-0.1,-.9) -- (10,0);
\draw (10,0) -- (.225,1.483);
\draw (10,0) -- (.225,-1.483);
\draw[dotted] (0,0) -- (.225,1.483);
\draw[dotted] (0,0) -- (.225,-1.483);
\draw[dotted] (0,0) -- (-3,0);
\draw (.5,1) node {$\epsilon/k$};
\draw (-2,-.3) node {$2\epsilon/k$};
\draw (6.7,0) node {$\alpha$};
\draw (5.2,0) node {$2\beta$};
\draw [red,thick,domain=177:185] plot ({10+ 3*cos(\x)}, {3*sin(\x)});
\draw [blue,thick,domain=171.5:188.5] plot ({10+ 4.5*cos(\x)}, {4.5*sin(\x)});
\filldraw (10,0) circle (1pt) node[align=right,   below] {$y_2$};
\filldraw (-0.5,.6) circle (1pt) node[align=right,   below] {$y_1$};
\filldraw (-0.1,-.9) circle (1pt) node[align=right,   below] {$\xi$};
\filldraw (0,0) circle (1pt) node[align=left,   below] {$x_j$};
\filldraw (10,0) circle (1pt) node[align=right,   below] {$y_2$} ;
\end{tikzpicture}
\caption{}
\label{pic:circles angles}
\end{figure}

In order to do this, we proceed as before as in the writing of $F(\delta)$ letting now $\zeta=\xi-y_2-\((\xi-y_2)\cdot \dfrac{e}{|e|}\)\dfrac{e}{|e|}$. Then
$\zeta\cdot \dfrac{e}{|e|}=0$, $(\xi-y_2)\cdot \dfrac{\zeta}{|\zeta|}=|\zeta|$, and
\begin{align*}
\xi-y_2&=\(\(\xi-y_2\)\cdot \dfrac{\zeta}{|\zeta|}\)\,\dfrac{\zeta}{|\zeta|}+\(\(\xi-y_2\)\cdot \dfrac{e}{|e|}\)\,\dfrac{e}{|e|}=
\sin \theta\,\dfrac{\zeta}{|\zeta|}+\cos \theta\,\dfrac{e}{|e|};
\end{align*}
recalling that $\theta$ is the angle between $\xi-y_2$ and $e$.
Then the left hand side of \eqref{eq:upper estimate of angle} can be written as
\begin{equation}\label{eq:definition of Gtheta}
G(\theta)=\text{angle}\(A(x,x_0;\xi,y_2)^{1/2} (\xi-y_2),  A(x,x_0;\xi,y_2)^{1/2}e\)
\end{equation}
and since $\theta>\pi/2$, it follows as in \eqref{eq:split formula for cos F} that
\[
\cos G(\theta)
=
-\dfrac{1+\bar B\, \tan \theta 
}{\sqrt{1+\bar C\, \tan^2 \theta
+
2\,\bar B\, \tan \theta}}
\]
with 
\[
\bar B=\dfrac{\left\langle 
A^{1/2}
\dfrac{\zeta}{|\zeta|}, A^{1/2}\dfrac{e}{|e|}\right\rangle}{\left| A^{1/2}\dfrac{e}{|e|}\right|^2},\qquad
\bar C
=
\dfrac{\left| A^{1/2}\dfrac{\zeta}{|\zeta|}\right|^2}{\left| A^{1/2}\dfrac{e}{|e|}\right|^2}.\]

Let us now analyze the function $G(\theta)$ when $\theta$ is close to $\pi$.
For $s<0$ and with $g$ defined as in \eqref{eq:definition of g}
but with $\bar B$ and $\bar C$ instead of $B$ and $C$, we have
\begin{align*}
\bar \Delta(s)=-1-g(s)
&=
-\dfrac{\(\bar C-\bar B^2\)\,s^2}{\(1+\bar B\,s+\sqrt{1+\bar C\,s^2 +2\,\bar B\,s}\)\sqrt{1+\bar C\,s^2 +2\,\bar B\,s}}.
\end{align*}
Applying the estimates for the last denominator obtained in Step 1 yields  
\[
\bar \Delta(s)\geq -2\,\(\bar C-\bar B^2\)\,s^2\qquad \text{for $-1/\(2\sqrt{\bar C}\)<s<0$}.
\]
In particular, 
since $\lambda/\Lambda\leq \bar C\leq \Lambda/\lambda$, and $|\bar B|\leq \sqrt{\bar C}$,
we obtain the bound
\[
\bar \Delta(s)\geq -2\,\dfrac{\Lambda}{\lambda}\,s^2\qquad \text{for $-1/\(2\sqrt{\bar C}\)<s<0$}.
\]
This implies that 
\[
-1-\cos G(\theta)\geq -2\,\dfrac{\Lambda}{\lambda}\,\tan^2 \theta\]
for $\theta$ such that $-(1/2)\sqrt{\lambda/\Lambda}<\tan \theta <0$.
Now $\tan \theta\sim \theta- \pi$ when $\theta\to \pi^-$, and so
\[
-1+8\,\dfrac{\Lambda}{\lambda}\,(\pi-\theta)^2\geq \cos G(\theta),
\]
for $\pi-\theta_1<\theta<\pi$ with $\theta_1>0$ small depending only on $\Lambda/\lambda$; which implies the following lower bound for the left hand side of \eqref{eq:upper estimate of angle}
\begin{equation}\label{eq:lower bound for Ftheta}
G(\theta)\geq \arccos\(-1+8\,\dfrac{\Lambda}{\lambda}\,(\pi-\theta)^2\),
\end{equation}
for all $\pi-\theta_1<\theta<\pi$.
\setstcolor{blue}
We now choose $\epsilon>0$ so that for each $\xi\in B_j=B_{\epsilon/k}(x_j)$ the angle $\theta=\text{angle} (\xi-y_2,e)$ satisfies $\pi-\theta_1<\theta<\pi$.
Recall once again that $x_0\in S_{kj}$, $y_1\in Tx_0\cap B_j$, and $y_2\in Tx_0$ with $|y_2-y_1|\geq 1/2k$. 
Then $|y_2-x_j|\geq |y_2-y_1|-|x_j-y_1|\geq \dfrac{1}{2k}-\dfrac{\epsilon}{k}
=\(\dfrac{1}{2\epsilon} -1\)\,\dfrac{\epsilon}{k}$;
that is, $y_2\notin B_{\(\frac{1}{2\epsilon} -1\)\,\frac{\epsilon}{k}}(x_j)$. For each $\delta'$ large there is $\epsilon$ sufficiently small such that $\dfrac{1}{2\epsilon} -1=1+\delta'$.
Let $\alpha$ be the angle between the vectors $\overrightarrow{y_2\xi}$ and $-e=\overrightarrow{y_2y_1}$, and consider the convex hull of $B_j$ and $y_2$, i.e., the ice-cream cone containing $B_j$. If $\beta$ is the angle between the vector $\overrightarrow{y_2x_j}$ and the edge of the convex hull, we have $\alpha \leq 2\beta$, see Figure \ref{pic:circles angles}. So $\sin \alpha\leq 2\sin \beta=2\dfrac{\epsilon/k}{|y_2-x_j|}\leq 2\frac{\epsilon/k}{\(\frac{1}{2\epsilon} -1\)\,\epsilon/k}=2/(1+\delta')$, a number than can be made arbitrarily small taking $\epsilon$ small, in particular, it can be made smaller than $\sin \theta_1$.
Therefore, with this choice of $\epsilon$, the ball $B_j$ is determined and the inequality \eqref{eq:lower bound for Ftheta} can be be applied when $\xi \in B_j$. 
Then combining \eqref{eq:upper estimate of angle}, \eqref{eq:definition of Gtheta}, and \eqref{eq:lower bound for Ftheta} we get that 
\[
\arccos\(-1+8\,\dfrac{\Lambda}{\lambda}\,(\pi-\theta)^2\)
\leq
\pi/2+C(\Lambda/\lambda)\,\text{angle} (x-x_0,e)
\]
for $\pi-\theta_1<\theta<\pi$. Since $x\in \Gamma$, $\text{angle} (x-x_0,e)\leq \delta_0$. Thus, if $\theta\to \pi^-$ this yields a contradiction since $\arccos (-1)=\pi$. 
The proof of 
\eqref{eq:intersection empty} is then complete.

{\bf Step 3.} We are now in a position to prove that $x_0\in S_{kj}$ cannot be a point of density for $S_{kj}$.
Recall $S_{kj}=S_k\cap B_j^*$ where $B_j^*=\{x\in \text{dom}(T):Tx\cap B_j\neq \emptyset\}$.
From \eqref{eq:intersection empty}, it follows that 
$B_j^*\cap \Gamma= \emptyset$ with $\Gamma$ the cone in Step 2.
Let $B_r(x_0)$ be the Euclidean ball centered at $x_0$ with radius $r$, then
\[
S_{kj}\cap B_r(x_0)=\(S_{kj}\cap B_r(x_0)\cap \Gamma\)\cup 
\(S_{kj}\cap B_r(x_0)\cap \Gamma^c\)
=
S_{kj}\cap B_r(x_0)\cap \Gamma^c,
\]
and therefore
\[
\dfrac{|S_{kj}\cap B_r(x_0)|_*}{|B_r(x_0)|}\leq \dfrac{|\Gamma^c\cap B_r(x_0)|}{|B_r(x_0)|}=c_{\delta_0} <1,
\]
for all $r$ and so from Lemma \ref{lm:densitypointslebesgue}, $x_0$ is not a point of density for $S_{kj}$.
This completes the proof of Theorem \ref{thm:monotoneHntoR2n+1}.
\end{proof}

\begin{remark}\rm
Under certain properties of the cost function $c$ and that the map $T$ is $c$-cyclically monotone, then it is proven in [GM96, Corollary 3.5] that $T$ is single-valued almost everywhere. 
\end{remark}

\subsection{Inverse maps}
Given a multivalued map $T:\R^n\to \mathcal P(\R^n)$ the domain of $T$ is the set $\text{dom}(T)=\{x\in \R^n: Tx\neq \emptyset\}$ and the range of $T$ the set $\text{ran}(T)=\bigcup_{x\in \R^n}Tx$.
The inverse of $T$ is the multivalued map $T^{-1}:\R^n\to \mathcal P(\R^n)$ defined by $T^{-1}y=\{x\in \R^n:y\in Tx\}$.
Clearly $\text{dom}(T^{-1})=\text{ran}(T)$.

If $T$ is $h$-monotone with $h$ even, then $T^{-1}$ is $h$-monotone and from Theorem \ref{thm:monotoneHntoR2n+1} $T^{-1}$ is single-valued a.e. and we obtain the following.

\begin{corollary}[of Aleksandrov type]\label{cor:aleksandrovlemma}
If $T:\R^n\to \mathcal P(\R^n)$ is a multivalued map that is $h$-monotone, with $h$ satisfying the assumptions of Theorem \ref{thm:monotoneHntoR2n+1}, then the set
\[
S=\left\{p\in \R^n: \text{there exist $x,y\in \R^n$, $x\neq y$, such that $p\in T(x)\cap T(y)$}\right\}
\]
has measure zero.
\end{corollary}
\begin{proof}
The map $T^{-1}:\R^n\to \mathcal P(\R^n)$ is $h$-monotone and from Theorem \ref{thm:monotoneHntoR2n+1}
is single valued a.e. Since $S=\{p\in \R^n:\text{$T^{-1}p$ is not a singleton}\}$, the corollary follows.
\end{proof}

Using Corollary \ref{cor:aleksandrovlemma}, we can define the push forward measure of an $h$-monotone map.
Let $f\in L^1_{\text{loc}}(\R^n)$, $f\geq 0$, and let $T:\R^n\to \mathcal P(\R^n)$ be an $h$-monotone map that is measurable, i.e., $T^{-1}(E)$ is a Lebesgue measurable subset of $\R^n$ for each $E\subset \R^n$ Lebesgue measurable.
Define
\[
\mu(E)=\int_{T^{-1}(E)}f(x)\,dx.
\]
Then $\mu$ is $\sigma$-additive. In fact, if $\{E_i\}_{i=1}^\infty$ are disjoint Lebesgue measurable sets, then it follows from Corollary \ref{cor:aleksandrovlemma} that $|T^{-1}(E_i)\cap T^{-1}(E_j)|=0$ for $i\neq j$; $T^{-1}(E)=\{x\in \R^n:T(x)\cap E\neq \emptyset \}$.

\subsection{Maximal $h$-monotone maps}
If $T:\R^n\to \mathcal P(\R^n)$ is a multivalued map, then the graph of $T$ is by definition $\text{graph}(T)=\{(x,y)\in \R^n\times \R^n: y\in Tx\}$. If $T_1,T_2$ are two multivalued maps, then $T_1\preceq T_2$ iff $\text{graph}(T_1)\subseteq \text{graph}(T_2)$, that is, for each $x\in \R^n$, $T_1x\subset T_2x$.
The relation $\preceq$ is a partial order on the set of all multivalued maps, and therefore the class of all multivalued maps with the relation $\preceq$ is a partially ordered set.
Let $\mathcal M_h$ denote the class of all multivalued maps that are $h$-monotone. Then the set $\(\mathcal M_h,\preceq\)$ is  inductive, that is, if $\mathcal S\subset \mathcal M_h$ is a totally ordered set or chain (that is, given $T_1,T_2\in \mathcal S$ then either $T_1\preceq T_2$ or $T_2\preceq T_1$), then $\mathcal S$ has a upper bound in $\mathcal M_h$, i.e., there exists $T\in \mathcal M_h$ such that $R\preceq T$ for all $R\in \mathcal S$.
Indeed, let $\tilde T$ be the map with $\text{graph}(\tilde T)=\bigcup_{T\in \mathcal S}\text{graph}(T)$, i.e., $\tilde Tx=\bigcup_{T\in \mathcal S}Tx$. Let us show that $\tilde T\in \mathcal M_h$. If $x,y\in \text{dom}(\tilde T)$, $\xi\in \tilde Tx$, and $\zeta\in \tilde Ty$, then there exists $T_1,T_2\in \mathcal S$ such that $\xi \in T_1x$ and $\zeta\in T_2y$. Since $\mathcal S$ is a chain, it follows that $T_1\preceq T_2$ or $T_2\preceq T_1$, that is,
$T_1z\subset T_2z$ or $T_2z\subset T_1z$ for all $z$. In particular, 
$\xi \in T_2x$ or $ \zeta \in T_1y$ and the $h$-monotonicity of $\tilde T$ follows from the $h$-monotonicity of $T_1$ or $T_2$.
Therefore $\(\mathcal M_h,\preceq\)$ is inductive and from Zorn's lemma, $\(\mathcal M_h,\preceq\)$ has a maximal element, i.e., there exists $T\in \mathcal M_h$ such that $R\preceq T$ for all $R\in \mathcal M_h$.

We say $T\in \mathcal M_h$ is maximal $h$-monotone if whenever $T'\in \mathcal M_h$ satisfies $T\preceq T'$ then we must have $T'=T$.

Given $T\in \mathcal M_h$, there exists $\tilde T\in \mathcal M_h$ such that $\tilde T$ is maximal $h$-monotone with $T\preceq \tilde T$. Indeed, consider the class $\mathcal R$ of all $R\in \mathcal M_h$ with $T\preceq R$. If $\mathcal S\subset \mathcal R$ is any chain, then proceeding as before the map defined by $\tilde Rx=\bigcup_{R\in \mathcal S}Rx$ is $h$-monotone and is an upper bound for $\mathcal S$. Therefore by Zorn's lemma, $\mathcal R$ has a maximal element.

We have the following characterization: The map $T\in \mathcal M_h$ is maximal $h$-monotone if given $(x,\xi)\in \R^n\times \R^n$ satisfying the condition
\[
h(x-\xi)+h(y-\zeta)\leq h(x-\zeta)+h(y-\xi)
\] 
for each $y\in \text{dom} (T)$ and for all $\zeta\in Ty$, then we must have $\xi\in Tx$.
Indeed, suppose there exists $(x_0,\xi_0)\in \R^n\times \R^n$ such that 
$h(x_0-\xi_0)+h(y-\zeta)\leq h(x_0-\zeta)+h(y-\xi_0)$ for all $y\in \text{dom} (T)$ and for all $\zeta\in Ty$, with $\xi_0\notin Tx_0$. If we define $T'x=Tx$ for $x\neq x_0$ and $T'x=Tx_0\cup \xi_0$ for $x=x_0$, then $T'$ is $h$-monotone and $\text{graph}(T)\subsetneqq  \text{graph}(T')$; so $T$ is not maximal.
Reciprocally, if the condition holds and $T\preceq T'$ with $T'\in \mathcal M_h$, then we want to prove that $T'x\subset Tx$ for all $x\in \R^n$. 
If $x\in \text{dom}(T)$, since $T\preceq T'$, then $x\in \text{dom}(T')$ and let $\xi\in T'x$. Since $T'$ is $h$-monotone, $h(x-\xi)+h(y-\zeta)\leq h(x-\zeta)+h(y-\xi)$ for all $y\in \text{dom}(T')$ and for all $\zeta\in T'y$. 
Since $\text{dom}(T)\subset \text{dom}(T')$, the last inequality holds for all $y\in \text{dom}(T)$ and for all $\zeta \in Ty$, and hence $\xi\in Tx$.
It remains to prove that if $Tx=\emptyset$, then $T'x=\emptyset$. Suppose $T'x\neq \emptyset$. Then there is $\xi\in T'x$, and since $T'$ is $h$-monotone we have 
$h(x-\xi)+h(y-\zeta)\leq h(x-\zeta)+h(y-\xi)$ for all $y\in \text{dom}(T')$ and for all $\zeta\in T'y$.
In particular, this holds for all $y\in \text{dom}(T)$ and for all $\zeta\in Ty$. Consequently, $\xi \in Tx$ and so $Tx\neq \emptyset$.

\begin{theorem}\label{eq:GcontinuousrelativetoHminusZ}
Let $T:\R^n\to \mathcal P(\R^n)$ be a multivalued map that is maximal $h$-monotone, 
%
and let $Z\subset \R^n$ be a set of measure zero
such that $T$ is single valued in $\R^n\setminus Z$.
Then $T$ is continuous relative to $\R^n\setminus Z$.
\end{theorem}
\begin{proof}
We first observe that since $T$ is $h$-monotone, then from the $L^\infty$-estimate \cite[Theorem 2.1]{Gutierrez-Montanari:Linfty-estimates} it follows that $T$ is locally bounded.
Let $x_0\in \R^n\setminus Z$ and suppose $T$ is discontinuous relative to $\R^n\setminus Z$ at $x_0$.
Then there exists $\delta>0$ such that for each $\epsilon>0$ there exists $x_\epsilon\in \R^n\setminus Z$ such that $|x_\epsilon-x_0|<\epsilon$
and $|p_\epsilon-p_0|>\delta$ with $p_0\in T(x_0)$ and $p_\epsilon\in T(x_\epsilon)$.
From the local boundedness of $T$, it follows that there exists a constant $M>0$ such that $|p_\epsilon|\leq M$ for all $\epsilon<1$.
Then there exists a subsequence $p_{\epsilon_j}\to p_1$ as $\epsilon_j \to 0$, and so $|p_1-p_0|\geq \delta$.
On the other hand, from the $h$-monotonicity
\[
h(x_{\epsilon_j}-p_{\epsilon_j})+h(x-p)\leq h(x_{\epsilon_j}-p)+h(x-p_{\epsilon_j})\qquad \forall p\in Tx.
\]
%
Letting $\epsilon_j\to 0$ yields
$h(x_0-p_1)+h(x-p)\leq h(x_0-p)+h(x-p_1)$ for all $p\in Tx$, and since $T$ is maximal, we get $p_1\in T(x_0)$, that is,
$T$ is not single valued at $x_0$, a contradiction.
\end{proof}

\begin{corollary}\label{cor:G(H^n)ismeasurable}
Under the assumptions of Theorem \ref{eq:GcontinuousrelativetoHminusZ}, we have that 
$T(\R^n)$ is a Lebesgue measurable set in $\R^n$.
\end{corollary}
\begin{proof}
Let $T_i$ be the $i$-th component of $T$. Let $Z$ be the set of measure zero such that $T$ is single valued in $\R^n\setminus Z$.
The function $T_i$ is continuous relative to $\R^n\setminus Z$ and therefore for each $\alpha$ the set
$\{x\in \R^n\setminus Z: T_i(x)\leq \alpha\}$ is relatively closed and so Lebesgue measurable.
\end{proof}

\begin{theorem}\label{thm:monge ampere measure for h monotone maximal}
Under the assumptions of Theorem \ref{eq:GcontinuousrelativetoHminusZ}, the class
\[
\Sigma=\left\{E\subset \R^n: T(E)\text{ is Lebesgue measurable} \right\}
\]
is a $\sigma$-algebra, and the set function
\[
\mu(E)=|T(E)|
\]
is $\sigma$-additive in $\Sigma$, that is, $(\R^n,\Sigma,\mu)$ is a measure space.
\end{theorem}
\begin{proof}
If $E_k\in\Sigma$, $k=1,2,\cdots$, then $T\left(\cup_{k=1}^\infty E_k\right)=\cup_{k=1}^\infty T(E_k)\in \Sigma$.
For any set $E\subset \R^n$ we have
\[
T(\R^n\setminus E)=\left[T(\R^n)\setminus T(E)\right]\cup \left[T(\R^n\setminus E)\cap T(E) \right].
\]
Suppose $E\in \Sigma$, then from Corollary \ref{cor:aleksandrovlemma}, the set $T(\R^n\setminus E)\cap T(E)$ has measure zero and then from Corollary \ref
{cor:G(H^n)ismeasurable} we obtain $\R^n\setminus E\in \Sigma$.
The proof of the $\sigma$-additivity is the same as the proof of \cite[Theorem 1.1.13]{Gut:book} with Lemma 1.1.12 there now replaced by 
Corollary \ref{cor:aleksandrovlemma} and $\partial u$ replaced by $T$.
\end{proof}

\subsection{On the rectifiability of $c$-monotone sets}\label{sec:rectifiability}\footnote{We like to thank Robert McCann for useful comments and for mentioning his paper \cite{McCPW} from where the argument in this section is derived.}

Let $c(x,y)$ be a $C^2$ cost from $\R^n\times \R^n$ to $\R$.
Denote by
\[
D^2_{xy}c(x,y)
\]
the $n\times n$ matrix having entries $\dfrac{\partial^2 c}{\partial x_i\partial y_j}(x,y)$.

Let $\Phi:\R^n\times \R^n\to \R^n\times \R^n$ be the Cayley transform defined by 
\[
\Phi(x,y)=\dfrac{1}{\sqrt2}\(x+y,x-y\).
\]

Let us fix a point $(x_0,y_0)$ and suppose that
\begin{equation}\label{eq:non singular point bis}
\det D^2_{xy}c(x_0,y_0)\neq 0,
\end{equation}
that is, the matrix $A:= -D^2_{xy}c(x_0,y_0)$ is non-singular.

Let us write
\[
c(x,y)=-(Ax)\cdot y+c(x,y)+(Ax)\cdot y
\]
and set $G(x,y)=c(x,y)+(Ax)\cdot y$. We have $D^2_{xy}\((Ax)\cdot y\)=A$,
so 
\[
D^2_{xy}c(x,y)=-A+D^2_{xy}G(x,y).
\]
Obviously, $D^2_{xy}G(x_0,y_0)=0$, and so for each $\e>0$ there is a convex neighborhood $N$ of 
$(x_0,y_0)$ such that $\|D^2_{xy}G\|_{L^\infty (N)}\leq \e$.


Let $S\subset \R^n\times \R^n$ be a $c$-monotone set, that is,  
for all points $(x,y),(x',y')\in S$ we have
\[
c(x,y)+c(x',y')\leq c(x,y')+c(x',y).
\]
In particular, for $(x,y),(x',y')\in S\cap N$ we can write 
that 
\begin{align*}
&-(Ax)\cdot y+G(x,y)-(Ax')\cdot y'+G(x',y')\\
&\qquad \leq -(Ax)\cdot y'+G(x,y')-(Ax')\cdot y+G(x',y).
\end{align*}
Rearranging terms yields
\begin{align*}
&(Ax'-Ax)\cdot (y-y')\leq G(x,y')-G(x,y)-\(G(x',y')-G(x',y)\).
\end{align*}
Now 
\begin{align*}
G(x,y')-G(x,y)&=\int_0^1 D_yG\(x,y+s(y'-y)\)\cdot (y'-y)\,ds\\
G(x',y')-G(x',y)&=\int_0^1 D_yG\(x',y+s(y'-y)\)\cdot (y'-y)\,ds,
\end{align*}
so
\begin{align*}
&G(x,y')-G(x,y)-\(G(x',y')-G(x',y)\)\\
&=\int_0^1 D_yG\(x,y+s(y'-y)\)\cdot (y'-y)\,ds
-\int_0^1 D_yG\(x',y+s(y'-y)\)\cdot (y'-y)\,ds\\
&=
\int_0^1\int_0^1 \langle D^2_{xy}G\(x'+t(x-x'),y+s(y'-y)\)(x-x'),y'-y\rangle\,dt ds.
\end{align*}
Hence
\[
|G(x,y')-G(x,y)-\(G(x',y')-G(x',y)\)|\leq \e\,|x-x'|\,|y-y'|
\]
and so 
\begin{equation}\label{eq:estimate with A and epsilon}
A(x'-x)\cdot (y-y')\leq \e\,|x-x'|\,|y-y'| \quad \forall (x,y),(x',y')\in S\cap N.
\end{equation}

Let
\[
\sqrt2\,u=Ax+y,\quad \sqrt2\,v=Ax-y,\quad \sqrt2\,u'=Ax'+y',\quad \sqrt2\,v'=Ax'-y';
\]
\[
\Delta x=Ax-Ax',\quad \Delta y=y-y',\quad \Delta u=u-u',\quad \Delta v=v-v'.
\]
Then 
\[
\Delta x+\Delta y=\sqrt2\,\Delta u,\qquad \Delta x-\Delta y=\sqrt2\,\Delta v,
\]
and so 
\[
\Delta u+\Delta v=\sqrt2\,\Delta x,\qquad \Delta u-\Delta v=\sqrt2\,\Delta y,
\]
From \eqref{eq:estimate with A and epsilon} we then have
\[
-\Delta x\cdot \Delta y\leq \e\,|x-x'|\,|\Delta y|=\e\,|A^{-1}A(x-x')|\,|\Delta y|
\leq \e\,\|A^{-1}\|\,|\Delta x|\,|\Delta y|,
\]
that is,
\[
\Delta x\cdot \Delta y\geq -\e\,\|A^{-1}\|\,|\Delta x|\,|\Delta y|.
\]
Now 
\begin{align*}
|\Delta u|^2-|\Delta v|^2
=
2\,\Delta x\cdot \Delta y&\geq -2\,\e\,\|A^{-1}\|\,|\Delta x|\,|\Delta y|\\
&=
-\,\e\,\|A^{-1}\|\,|\Delta u-\Delta v|\,|\Delta u+\Delta v|\\
&\geq 
-\,\e\,\|A^{-1}\|\,\(|\Delta u|^2+|\Delta v|^2\),
\end{align*}
because $|a-b|^2|a+b|^2\leq \(|a|^2+|b|^2\)^2$ for any vectors $a,b$.
We then get
\begin{equation}\label{eq:Lipschitz continuity}
|\Delta v|\leq \sqrt{\dfrac{1+\e\,\|A^{-1}\|}{1-\e\,\|A^{-1}\|}}\,|\Delta u|,
\end{equation}
for $\epsilon<1/\|A^{-1}\|$. 
This shows in particular that if $u=u'$ then $v=v'$.

Now set $\Psi(x,y)=\Phi(Ax,y)=\dfrac{1}{\sqrt2}\(Ax+y,Ax-y\)=\(\Psi_1(x,y),\Psi_2(x,y)\)$. Then $\Psi:\R^n\times \R^n\to \R^n\times \R^n$ is an invertible linear transformation and in particular $\Psi:S\cap N\to \Psi(S\cap N)\subset \R^n\times \R^n$ is a bijection. Let $\Pi_1(x,y)=x$ and $\Pi_2(x,y)=y$ be the coordinate projections.
From \eqref{eq:Lipschitz continuity}, given $u\in \Pi_1\(\Psi(S\cap N)\)$ there is a unique $v\in \Pi_2\(\Psi(S\cap N)\)$  such that $(u,v)\in \Psi(S\cap N)$ so if we define $F(u)=v$, then 
\[
F:\Pi_1\(\Psi(S\cap N)\)\to \Pi_2\(\Psi(S\cap N)\)
\] 
is a Lipschitz function. 
The graph of $F$, $G(F)$, equals to $\Psi(S\cap N)$, and since $\Psi$ is a bijection we get $S\cap N=\Psi^{-1}(G(F))$; $\Psi^{-1}(u,v)=\(A^{-1}\(\dfrac{1}{\sqrt2}(u+v)\),\dfrac{1}{\sqrt2}(u-v)\)$. That is, $S\cap N$ is the image of a Lipschitz graph by a linear transformation.

This proves that if $(x_0,y_0)$ is a point satisfying \eqref{eq:non singular point bis}, and $S$ is a $c$-monotone set, then there is a neighborhood $N$ of $(x_0,y_0)$ such that the set $S\cap N$ is the image of a Lipschitz graph by a linear transformation. 
Notice that this last fact does not imply in general that $Tx$ is single valued a.e.. In fact, if $f:\R^n\to \R$ is a smooth function and we define the multivalued map $Tx=\{f(x)+k:k\in \Z\}$, then the set $S=\{(x,Tx):x\in \R^n\}$ is union of smooth graphs but the map $T$ is not single valued at any point.
In our case if we consider a multivalued map $T$ satisfying
\[
c(x,\xi)+c(y,\zeta)\leq c(x,\zeta)+c(y,\xi)\quad \forall \xi\in Tx,\zeta \in Ty
\]
and define the set
$
S=\{(x,\xi):\xi\in Tx; x\in \text{dom\,}T \}$, it follows that $S$ is a $c$-monotone set.

The case left is when $(x_0,y_0)$ does not satisfy \eqref{eq:non singular point bis}.
For our cost $c(x,y)=h(x-y)$ with $h$ homogeneous of degree $p\geq 2$ and having Hessian positive definite in the unit sphere, \eqref{eq:non singular point bis} may not happen only when $x_0=y_0$.
Let $D=\{(x,x):x\in \R^n\}$ be the diagonal. Write 
$S=\(S\cap D\)\cup \(S\cap D^c\)$. If $(x_0,y_0)\in S\cap D^c$, then $(x_0,y_0)$ satisfies \eqref{eq:non singular point bis} and therefore there is a neighborhood $N$ of $(x_0,y_0)$ such that $S\cap D^c\cap N$ is the image of a Lipschitz graph.
On the other hand, the map $T:\R^n\to \R^n\times \R^n$ given by $Tx=(x,x)$ is Lipschitz and 
$S\cap D=T\(\Pi_1(S\cap D)\)$, that is, $S\cap D$ is $n$-rectifiable in the sense of 
\cite[Sect. 3.2.14, Definition (1)]{federerbook}.

 


\section{Appendix}\label{sec:appendix}
\setcounter{equation}{0}

We conclude the paper with the following lemma that yields the differentiability property we use in the proof of Theorem  \ref{thm:monotoneHntoR2n+1}.

\begin{lemma}\label{lm:densitypointslebesgue}
Let $S\subset \R^n$ be a set not necessarily Lebesgue measurable and consider
\[
f(x):=\limsup_{r\to 0} \dfrac{|S\cap B_r(x)|_*}{|B_r(x)|},
\]
where $|\cdot |_*$ and $|\cdot |$ denote the Lebesgue outer measure and Lebesgue measure respectively.
If
\[
M=\{x\in S: f(x)<1\},
\]
then $|M|=0$.
Here $B_r(x)$ is the Euclidean ball centered at $x$ with radius $r$. 

Moreover, if $B_r(x)$ is a ball in a metric space $X$ and 
$\mu^*$ is a Carath\'eodory outer measure on $X$\footnote{From \cite[Theorem (11.5)]{wheeden-zygmund:book} every Borel subset of $X$ is Carath\'eodory measurable.}, then a similar result holds true for all $S\subset X$.
\end{lemma}
\begin{proof}
Fix $x\in M$. There exists a positive integer $m$ such that $f(x)<1-\dfrac{1}{m}<1$, and let $m_x$ be the smallest integer with this property.
So for each $\eta>0$ sufficiently small we have
\begin{equation}\label{densityinequality}
\sup_{0<r\leq \delta}\dfrac{|S\cap B_r(x)|_*}{|B_r(x)|}<1-\dfrac{1}{m_x},\qquad \text{for all $0<\delta \leq \eta$}.
\end{equation}
Given a positive integer $k$, let $M_k=\{x\in M: m_x=k\}$. We have $M=\cup_{k=1}^\infty M_k$.
We shall prove that $|M_k|=0$ for all $k$.
Suppose by contradiction that $|M_k|_*>0$ for some $k$, we may assume also that $|M_k|_*<\infty$.
Let us consider the family of balls $\mathcal F=\{B_r(x)\}_{x\in M_k}$ with $B_r(x)$ satisfying \eqref{densityinequality}. Then we have that the family $\mathcal F$ covers $M_k$ in the Vitali sense, i.e., for every $x\in M_k$ and for every $\eta>0$ there is ball in $\mathcal F$ containing $x$ whose diameter is less than $\eta$. Therefore from \cite[Corollary (7.18) and equation (7.19)]{wheeden-zygmund:book} we have that given $\epsilon>0$ there exists a family of disjoint balls $B_1,\cdots ,B_N$ in $\mathcal F$ such that
\begin{align*}
|M_k|_* -\epsilon &< \left| M_k\cap \cup_{i=1}^N B_i \right|_*,\qquad \text{and}\\
\sum_{i=1}^N |B_i| &< (1+\epsilon)|M_k|_*.
\end{align*}
Since $M_k\subset S$, then from \eqref{densityinequality} we get
\[
|M_k|_* -\epsilon
< \left| S\cap \cup_{i=1}^N B_i \right|_*
\leq \sum_{i=1}^N |S\cap B_i|_*\leq \left(1-\dfrac{1}{k}\right)\sum_{i=1}^N |B_i|
<(1+\epsilon)\left(1-\dfrac{1}{k}\right)|M_k|_*,
\]
then letting $\epsilon\to 0$ we obtain a contradiction.
\end{proof}


%

\end{document}